\documentclass[12pt]{article}
\usepackage{graphicx}
\usepackage{graphicx}
\usepackage{color}
\usepackage{amsfonts}
\usepackage{amssymb}
\usepackage{amsthm}
\usepackage[T1]{fontenc}

\usepackage{newlfont}
\usepackage{amsmath}
\usepackage{t1enc}

\usepackage[english]{babel}
\newtheorem{thm}{Theorem}[section]

\newtheorem{lem}[thm]{Lemma}
\newtheorem{prop}[thm]{Proposition}

\theoremstyle{definition}
\newtheorem{defn}[thm]{Definition}
\theoremstyle{remark} \numberwithin{equation}{section}
\def\proof{\noindent{\bf Proof }}

{}

\usepackage{amssymb}
\usepackage{relsize}
\usepackage{layout}
\definecolor{lightgray}{gray}{0.5}
\title{$\mathbb{D}$-solutions of BSDEs with Poisson jumps}

\date{ }
\begin{document}
\author{ I. Hassairi$^*$  }

\maketitle { $^*$ Facult\'e des Sciences de Sfax, département de mathématiques, Université de Sfax, Tunisie.\\ 
e-mail: imen.hassairi@yahoo.fr}

\begin{abstract} In this paper, we study backward stochastic differential equations (BSDEs shortly) with jumps that have Lipschitz generator in a general filtration supporting a Brownian motion and an independent Poisson random measure. Under just integrability on the data we show that such equations admits a unique solution which belongs to class $\mathbb{D}$.
\end{abstract}
\vspace{5mm}


\vspace{5mm}
\section{Introduction}

The notion of non-linear BSDEs was introduced by Pardoux and Peng (\cite{PP}, $1990$). These equations have been well studied because they are connected with a lot of applications especially in mathematical finance, stochastic control, partial differential equations, ...\\

Tang and Li \cite{LT} added into the BSDE a jump term that is driven by a Poisson random measure independent of the Brownian motion. The authors obtained the existence and uniqueness of a solution to such an equation when the terminal condition is square integrable and the generator is Lipschitz continuous w.r.t. the variables. Since then, a lot of papers (one can see \cite{KP,P,QS}) studied BSDEs with jumps due to the connections of this subject with mathematical finance and stochastic control.\\

Later, Situ Rong \cite{R} proved an existence result when the terminal time is a bounded random stopping time and the coefficient is non-Lipschitzian.\\

Recently, Song Yao analyzes in his work \cite{Y} the BSDEs with jumps with unbounded random time horizon and under a non-Lipschitz generator condition. He showed the existence and uniqueness of an $\mathbb{L}^p$-solution when the terminal condition is $p$-integrable, for any $p\in(1,\infty)$. For a given $V\in\mathcal{L}^2$, unlike the Brownian stochastic integrals case, the Burkholder-Davis-Gundy inequality is not applicable. So in his paper, he genralized the Poisson stochastic integral for a random field $V\in\mathcal{L}^p$.\\
 
In our paper, we investigate the existence and uniqueness of $\mathbb{D}$-solution for BSDEs when the noise is driven by a Brownian motion and an independent random Poisson measure. This paper generalizes the results of Briand et al. \cite{BDHPS} and Song Yao \cite{Y}. We suppose that $f$ is Lipschitz. Concerning the data, we assume that an integrability condition is hold.\\
 
 Our motivation for writing this paper is because it is the first step in the study of our forthcoming work on the existence and uniqueness of $\mathbb{D}$-solutions for reflected BSDEs with jumps. Finally, to our knowledge there is no such result in the literature.\\
  
The outline of this article is as follows: the following section contains all the notations and useful assumptions for the rest of the paper, in section 3, we showed uniqueness result and the essential estimates. Section 4 is devoted to the case where the data are in $\mathbb{L}^p$, for $p\in(1,2)$ then we treat the case $p=1$ and we showed the desired existence result.

\section{Notations and assumptions}
Let $(\Omega,\mathcal{F},(\mathcal{F}_t)_{t\leq T} )$ be a stochastic basis such that $(\mathcal{F}_t)_{t\leq T}$ is right continuous increasing family of complete sub $\sigma$-algebras of $\mathcal{F}$ and $(\mathcal{F}_0$ contains $\mathcal{N}$ the set of all $\mathbb{P}$-null sets of $\mathcal{F}$, $\mathcal{F}_{t^+}=\cap_{\epsilon>0}\mathcal{F}_{t+\epsilon}=\mathcal{F}_{t}$, $\forall t\leq T$. We assume that $(\mathcal{F}_t)_{t\leq T}$ is supported by the two mutually independent processes:\\

(i) $B=(B_t)_{0\leq t\leq T}$ be a standard $d$-dimensional Brownian motion.\\

(ii) let $\mathfrak{p}$ be an $U$-valued Poisson point process on $(\Omega,\mathcal{F},\mathbb{P}$, for some finite random Poisson measure $\mu$ on $\mathbb{R}^+\times U$ where $U\subset \mathbb{R}^m\setminus\{0\}$, the counting measure $\mu(dt,de)$ of $\mathfrak{p}$ on $[0,T]\times U$ has the compensator $\mathbb{E}[\mu(dt,de)]=\lambda(de)dt$. The corresponding compensated Poisson random measure $\tilde{\mu}(dt,de):=\mu(dt,de)-dt\lambda(de)$ is a martingale w.r.t. $\mathcal{F}$. The measure $\lambda$  is $\sigma$-finite on $U$ satisfying

\begin{equation}\nonumber
\int_{ U}(1\wedge|e|^2)\lambda(de)<+\infty.
\end{equation}

In this paper, let $\tilde{\mathcal{P}}$ denote the $\sigma$-algebra of $\mathcal{F}_t$-predictable sets on $\Omega\times [0,T]$. In addition, we assume that 
\begin{equation}\nonumber
\mathcal{F}_t=\sigma[\int_{A\times[0,s]}\mu(ds,de),s\leq t, A\in U] \cup \sigma[B_s, s\leq t] \cup \mathcal{N}.
\end{equation}
 
 For a given adapted càdlàg process $(X_t)_{t\leq T}$ and for any $t\leq T$ we set $X_{t^-}=\lim_{s\nearrow t} X_s$ with the convention that $X_{0^-}=X_0$ and $\Delta X_t=X_t-X_{t^-}$. For any scenario $\omega\in\Omega$, let $D_{\mathfrak{p}(\omega)}$ collect all jump times of the path $\mathfrak{p}(\omega)$ which is a countable subset of $(0,T]$.\\
 
 Let us introduce the following spaces of processes and notations considered in this work, for all $p> 0$,\\
 
 \begin{itemize}
 \item  $\mathcal{S}^p$ is the space of $\mathbb{R}$-valued, $\mathcal{F}_t$-adapted and càdlàg processes $(X_t)_{t\in[0,T]}$ such that
$$\left\|X\right\|_{\mathcal{S}^p}=\mathbb{E}\bigg[\sup_{t\leq T}\left|X_t\right|^p\bigg]^{  \frac{1}{p}}<+\infty.$$
If $p\geq 1$, $\|.\|_{\mathcal{S}^p}$ is a norm on $\mathcal{S}^p$ and if $p\in (0,1)$, $(X,X')\mapsto \left\|X-X'\right\|_{\mathcal{S}^p}$ is a distance on $\mathcal{S}^p$. Under this metric, $\mathcal{S}^p$ is complete.\\

\item $\mathcal{M}^p$ denotes the set of $\mathbb{R}^n$-valued and $\mathcal{F}_t$-predictably measurable processes $(X_t)_{t\in[0,T]}$   such that 
$$\left\|X\right\|_{\mathcal{M}^p}=\mathbb{E}\bigg[ \bigg(\int_{0}^{T}\left|X_s\right|^2 ds\bigg)^{p/2}\bigg]^{ \frac{1}{p}}<+\infty.$$

For $p\geq 1$, $\mathcal{M}^p$ is a Banach space endowed with this norm and for $p \in (0,1)$, $\mathcal{M}^p$ is a complete metric space with the resulting distance. For all $\beta\in(0,1]$ let us define $\mathcal{M}^\beta$ as the set of $\mathcal{F}_t$-progressively measurable processes $(X_t)_{t\in[0,T]}$ with values in $\mathbb{R}^d$ such that $$\left\|X\right\|_{\mathcal{M}^\beta}=\mathbb{E}\bigg[\bigg(\int_{0}^{T}\left|X_s\right|^2 ds\bigg)^{\beta/2}\bigg] <+\infty.$$
 
We denote by $\mathcal{M}^0$ the set of $\mathcal{P}$-measurable processes $Z:=(Z_t)_{t\leq T}$ with values in $\mathbb{R}^d$ such that $\int_0^T|Z_s(\omega)|^2ds<\infty$, $\mathbb{P}-a.s.$.\\

\item $\mathcal{L}^p_{\text{loc}}$ is the space of all $\mathcal{\tilde{P}}\otimes\mathcal{B}(U)$-measurable mappings $V:\Omega\times [0,T]\times U\rightarrow \mathbb{R}$ such that $\int_0^T \int_{U} |V_s(e)|^p\lambda(de)ds<\infty$. Let $\mathcal{L}^p$ be the set of all $V\in\mathcal{L}^p_{\text{loc}}$ such that $\Vert V\Vert_{\mathcal{L}^p}:= \bigg(\mathbb{E}\int_0^T\int_U|V_s(e)|^p\lambda(de)ds\bigg)^\frac{1}{p}<+\infty$.
 
 \end{itemize}
The stochastic integral with respect to the compensated Poisson random measure $\tilde{\mu}(dt,de)$ is usually defined for locally square integrable random mappings $V\in\mathcal{L}^2_{\text{loc}}$. We recall, in the following lemma, a generalization of Poisson stochastic integrals for random mappings in $\mathcal{L}^p_{\text{loc}}$ for $p\in[1,2)$. For more details on the proof we refer the reader to Lemma 1.1 in \cite{Y}.

\begin{lem}\label{Y1}
  Let $p\in[1,2)$, we assign $M$ as the Poisson stochastic integral  
\begin{eqnarray}\label{Y}
\int_{[0,T]}\int_U V_s(e)\tilde{\mu}(ds,de),
\end{eqnarray}  
  
for any $V\in\mathcal{L}^p$. Analogous to the classic extension of Poisson stochastic integrals from $\mathcal{L}^2$ to $\mathcal{L}^2_{\text{loc}}$, one can define the stochastic integral (\ref{Y}) for any $V\in\mathcal{L}^p_{\text{loc}}$, which is a càdlàg local martingale with quadratic variation $$[M,M]_t=\int_0^t\int_U |V_s(e)|^2\mu(ds,de)$$
and whose jump process satisfies $\Delta M_t(\omega)=\mathbf{1}_{t\in D_{\mathfrak{p}(\omega)}}V(\omega,t,\mathfrak{p}(\omega))$, $\forall t\in(0,T]$. This generalized Poisson stochastic integral is still linear in $V\in\mathcal{L}^p_{\text{loc}}$.
\end{lem}

 The above lemma will be useful and plays a crucial role in the rest of the paper.\\

Let us recall that a process $X$ belongs to the class $\mathbb{D}$ if the family of  random variables $\{X_\tau,\,\tau\in \mathcal{T} \}$ is uniformly integrable with $\mathcal{T}$ is the set of all $\mathcal{F}_t$-stopping times $\tau \in [0,T],\,\mathbb{P}-a.s.$ We say that a sequence $(\tau_k)_{k\in\mathbb{N}}\subset T$ is stationary if $\mathbb{P}(\liminf_{k\rightarrow\infty}\lbrace\tau_k=T\rbrace)=1$. In (\cite{DM2}, pp.$90$) it is observed that the space of continuous (càdlàg), adapted processes from class $\mathbb{D}$ is complete under the norm $$\left\|X\right\|_\mathbb{D}=\sup_{\tau \in \mathcal{T}}\mathbb{E}[\left|X_\tau\right|].$$

In this paper, we consider the following assumptions:\\

(A1) A terminal value $\xi$ which is an $\mathbb{R}$-valued, $\mathcal{F}_T$-measurable random variable such that $\mathbb{E}|\xi|]<\infty$ ; \\

(A2) A random function $f:[0,T]\times\Omega\times\mathbb{R}\times\mathbb{R}^d\times\mathcal{L}  \rightarrow \mathbb{R}$ which with $(t,\omega,y,z,v)$ associates $f(t,\omega,y,z,v)$ and which is $\tilde{\mathcal{P}}\otimes\mathcal{B}(\mathbb{R}^{1+d})\otimes\mathcal{B}(\mathcal{L})$-measurable. In addition we assume: \\

(i) the process $(f(t,0,0,0))_{t\leq T}$ is $d\mathbb{P}\otimes dt$-integrable, i.e., $\mathbb{E}\bigg[\int_0^T|f(s,0,0,0)|ds\bigg]<\infty$;\\

(ii) $f$ is uniformly Lipschitz in $(y,z,v)$, i.e., there exists a constant $\kappa\geq 0$ such that for any $t\in[0,T]$, $y,y'\in\mathbb{R}$, $z,z'\in\mathbb{R}^d$ and $v, v'\in\mathcal{L}^p$ we have
$$\mathbb{P}-a.s,\ \left|f(t,\omega,y,z,v)-f(t,\omega,y',z',v')\right|\leq \kappa(\left|y-y'\right|+\left|z-z'\right|+\parallel v-v' \parallel);$$

(iii) $\mathbb{P}-a.s., \forall r>0,\ \int_0^T \psi_r(s)ds<+\infty$, where $\psi_r(t):=\sup_{\left|y\right|\leq r} \left|f(t,y,0,0)-f(t,0,0,0)\right|$.\\

(iv) There exist two constants $\gamma\geq 0$, $\alpha\in(0,1)$ and a non-negative progressively measurable process $g$ such that $\mathbb{E}[\int_0^T g_s ds]<\infty$ and 
\begin{center}$|f(t,y,z,v)-f(t,y,0,0)|\leq \gamma (g_t+|y|+|z|+\Vert v\Vert_{\mathcal{L}^p})^\alpha,t\in[0,T],y\in\mathbb{R}\ and\ z\in\mathbb{R}^d.$
\end{center}
Note that if $f$ does not depend on $z$ and $v$, the latter assumption is satisfied.\\

To begin with, let us now introduce the notion of $\mathbb{D}$-solutions of BSDEs with jumps which we consider throughout this paper.
\begin{defn} A triplet of processes $(Y,Z,V):=(Y_t,Z_t,V_t)_{t\leq T}$ with values in $\mathbb{R}^{1+d}\times\mathcal{L}^1$ is called a solution of the BSDE with jumps associated with $(f,\xi)$ if the following holds:  
\begin{eqnarray}\label{s2}\left\{
\begin{array}{l}
    Y\in\mathbb{D}, Z\in \mathcal{M}^0 \mbox{ and } V\in\mathcal{L}^1; \\\\
    		
    Y_t=\xi+\int_{t}^{T}f(s,Y_s,Z_s,V_s)ds -\int_{t}^{T}Z_s dB_s-\int_t^T \int_U V_s(e) \tilde{\mu}(ds,de), t\leq T;\\\\

\end{array}
\right.
\end{eqnarray}

The solution has jumps which arise naturally since the noise contains a random Poisson measure part.\\
\end{defn}
We are now going to prove the uniqueness of the solution for the (\ref{s2}) under the above assumptions on $f$ and $\xi$.

\section{Uniqueness and existence of a solution}

\subsection{A priori estimates}
\begin{lem} \label{L}
Let $(Y,Z,V)$ be a solution to BSDE (\ref{s2}) and assume that for $p>0$, $(\int_0^T f(s,0,0,0)ds)^p$ is integrable. If $Y\in\mathcal{S}^p$ then $Z\in\mathcal{M}^p$ , $ V\in\mathcal{L}$ and there exists a constant $C_{p,\kappa}$ such that ,
\begin{eqnarray}\label{L1}
   \mathbb{E}\bigg[\bigg(\int_0^T|Z_s|^2ds\bigg)^{\frac{p}{2}}+ \int_0^T\int_U |V_s(e)|^p\lambda(de)ds \bigg]  \leq C_{p,\kappa} \mathbb{E}\bigg[\sup_{t}|Y_t|^p+\bigg(\int_0^T |f(s,0,0,0)| ds\bigg)^p\bigg].
 \end{eqnarray}
\end{lem}

\proof Since there is a lack of integrability of the processes $(Y,Z,V)$, we will proceed by localization. For each integer $n$, let us define the stopping time 
$$\tau_n=\inf\{t\geq 0;\int_0^t|Z_s|^2 ds+\int_0^t\int_{U}|V_s(e)|^p   \lambda(de)ds> n\}\wedge T.$$

The sequence $(\tau_n)_{n\geq 0}$ is non-decreasing and converges to $T$. Using It\^o's formula with $|Y|^2$ on $[t\wedge\tau_n,\tau_n]$ we obtain,

\begin{eqnarray}\label{I4}
\begin{split}
& |Y_{t\wedge\tau_n}|^2+\int_{t\wedge\tau_n}^{\tau_n}|Z_s|^2 ds+\int_{t\wedge\tau_n}^{\tau_n}\int_{U} |V_s(e)|^2 \mu(ds,de)  = |Y_{ \tau_n} |^2 +2 \int_{t\wedge\tau_n}^{\tau_n}Y_s f(s,Y_s,Z_s,V_s) ds\\
&\qquad -2\int_{t\wedge\tau_n}^{\tau_n}Y_{s} Z_s dB_s -2\int_{t\wedge\tau_n}^{\tau_n}\int_{U}Y_{s^-} V_s(e) \tilde{\mu}(ds,de) .
\end{split}
\end{eqnarray}

But from the assumption on $f$, Young's inequality (for $\epsilon>0,\ ab\leq\frac{a^p}{p\epsilon^p}+\frac{\epsilon^q b^q}{q}$, with $\frac{1}{p}+\frac{1}{q}=1$) and the inequality ($2ab\leq 2a^2+\frac{b^2}{2}$) where $a,b\in\mathbb{R}$, we have

 $$\begin{array}{ll}
 \begin{split}
&  2  Y_s f(s,Y_s,Z_s,V_s)\leq 2\kappa |Y_s|^2+2\kappa |Y_s||Z_s|+2\kappa |Y_s| \Vert V_s\Vert_{\mathcal{L}^p}+2|Y_s| f(s,0,0,0)\\
& \qquad \leq 2\kappa(1+\kappa) |Y_s|^2+\frac{1}{2}|Z_s|^2+\frac{p-1}{p}\epsilon^{-\frac{p}{p-1}}(2\kappa)^\frac{p}{p-1}|Y_s|^\frac{p}{p-1}+ \frac{\epsilon^p}{p} \Vert V_s\Vert_{\mathcal{L}^p}^p+2|Y_s|f(s,0,0,0). 
  \end{split}
\end{array}$$

Plugging the above inequality into (\ref{I4}), we get

 $$\begin{array}  {ll}
\begin{split}
&  \frac{1}{2}\int_{t\wedge\tau_n}^{\tau_n}|Z_s|^2 ds+ \int_{t\wedge\tau_n}^{\tau_n}\int_{U} |V_s(e)|^2 \mu(ds,de)\\
&\qquad \leq  2\kappa(1+\kappa)\sup_{s\in[0,T]}|Y_s|^2+ \frac{p-1}{p}\epsilon^{-\frac{p}{p-1}} (2\kappa)^\frac{p}{p-1}\sup_{s\in[0,T]} |Y_s|^\frac{p}{p-1} \\
& \qquad +2\sup_{s\in[0,T]}|Y_s|\int_{t\wedge\tau_n}^{\tau_n}f(s,0,0,0)ds+ \epsilon^p\int_{t\wedge\tau_n}^{\tau_n}\int_{U} |V_s(e)|^p\lambda(de)ds\\
& \qquad +2\bigg|\int_{t\wedge\tau_n}^{\tau_n}Y_{s} Z_s dB_s\bigg|+2\bigg|\int_{t\wedge\tau_n}^{\tau_n}\int_{U}Y_{s^-} V_s(e) \tilde{\mu}(ds,de)\bigg|.
\end{split}
\end{array} $$

Hence, using the inequality ($2ab\leq  a^2+  b^2$), we obtain

 $$\begin{array}  {ll}
\begin{split}
&  \frac{1}{2}\int_{t\wedge\tau_n}^{\tau_n}|Z_s|^2 ds+ \int_{t\wedge\tau_n}^{\tau_n}\int_{U} |V_s(e)|^2 \mu(ds,de)\\
&\qquad \leq (2\kappa(1+\kappa)+1) \sup_{s\in[0,T]}|Y_s|^2+ \frac{p-1}{p}\epsilon^{-\frac{p}{p-1}} (2\kappa)^\frac{p}{p-1}\sup_{s\in[0,T]} |Y_s|^\frac{p}{p-1}\\
& \qquad + \bigg(\int_{t\wedge\tau_n}^{\tau_n}f(s,0,0,0)ds\bigg)^2 + \epsilon^p\int_{t\wedge\tau_n}^{\tau_n}\int_{U} |V_s(e)|^p\lambda(de)ds+2\bigg|\int_{t\wedge\tau_n}^{\tau_n}Y_{s} Z_s dB_s\bigg|\\
& \qquad +2\bigg|\int_{t\wedge\tau_n}^{\tau_n}\int_{U}Y_{s^-} V_s(e) \tilde{\mu}(ds,de)\bigg|.
\end{split}
\end{array} $$

Therefore

  \begin{eqnarray}  \label{e2}
 \begin{split}
& \bigg(\int_{t\wedge\tau_n}^{\tau_n}|Z_s|^2 ds\bigg)^\frac{p}{2}+\bigg( \int_{t\wedge\tau_n}^{\tau_n}\int_{U} |V_s(e)|^2 \mu(ds,de)\bigg)^\frac{p}{2}\\
&\qquad \leq C_{p,\kappa} \sup_{s\in[0,T]}|Y_s|^p+ C_{p,\kappa,\epsilon}\sup_{s\in[0,T]} |Y_s|^\frac{p^2}{2(p-1)}+ \bigg(\int_{t\wedge\tau_n}^{\tau_n}f(s,0,0,0)ds\bigg)^p\\
& \qquad + \epsilon^\frac{p^2}{2}\bigg(\int_{t\wedge\tau_n}^{\tau_n}\int_{U} |V_s(e)|^p\lambda(de)ds\bigg)^\frac{p}{2}+\bigg|\int_{t\wedge\tau_n}^{\tau_n}Y_{s} Z_s dB_s\bigg|^\frac{p}{2}+\bigg|\int_{t\wedge\tau_n}^{\tau_n}\int_{U}Y_{s^-} V_s(e) \tilde{\mu}(ds,de)\bigg|^\frac{p}{2} .
\end{split}
\end{eqnarray} 
 
On the other hand using BDG and Young inequalities we get,

 $$\begin{array}{ll}
   \begin{split} 
& \mathbb{E}\bigg[\bigg|\int_{t\wedge\tau_n}^{\tau_n}Y_{s} Z_s dB_s\bigg|^{p/2}\bigg]\leq c_{1p}\mathbb{E}\bigg[\bigg(\int_0^{\tau_n}|Y_{s}|^2 |Z_s|^2ds\bigg)^{p/4}\bigg]\\
&\qquad\qquad\qquad \leq c_{1p}\mathbb{E}\bigg[\bigg( \sup_{s\in[0,T]}|Y_s|\bigg)^{p/2} \bigg(\int_{t\wedge\tau_n}^{\tau_n}|Z_s|^2ds\bigg)^{p/4}\bigg]\\
&\qquad\qquad\qquad \leq \frac{c_{1p}^2}{2} \mathbb{E}\bigg(\sup_{s\in[0,T]}|Y_s|^p\bigg)+\frac{1}{2} \mathbb{E}\bigg(\int_{t\wedge\tau_n}^{\tau_n}|Z_s|^2ds\bigg)^{p/2}
  \end{split}
\end{array}$$

and similarly for the Poisson stochastic part which is uniformly integrable martingale by Lemma \ref{Y1},

 $$\begin{array}{ll}
 \begin{split} 
& \mathbb{E}\bigg[\bigg|\int_{t\wedge\tau_n}^{\tau_n}\int_{U}Y_{s^-} V_s(e) \tilde{\mu}(ds,de) \bigg|^{p/2}\bigg]\leq c_{2p} \mathbb{E}\bigg[\bigg(\int_{t\wedge\tau_n}^{\tau_n}\int_{U}|Y_{s}|^2| V_s(e)|^2 \mu(ds,de)\bigg)^{p/4}\bigg]\\
&\qquad\qquad\qquad \leq \frac{c_{2p}^2}{2} \mathbb{E}\bigg(\sup_{s\in[0,T]}|Y_s|^p\bigg)+\frac{1}{2} \mathbb{E}\bigg(\int_{t\wedge\tau_n}^{\tau_n} \int_{U}| V_s(e)|^2 \mu(ds,de)\bigg)^{p/2}
\end{split}
\end{array}$$

where $c_{1p}$ and $c_{2p}$ are real constants. Coming back to (\ref{e2}) and then taking expectation, we obtain, 

   \begin{eqnarray}   
 \begin{split}
& \frac{1}{2}\mathbb{E}\bigg[\bigg(\int_{t\wedge\tau_n}^{\tau_n}|Z_s|^2 ds\bigg)^\frac{p}{2}+\bigg( \int_{t\wedge\tau_n}^{\tau_n}\int_{U} |V_s(e)|^2 \mu(ds,de)\bigg)^\frac{p}{2}\bigg]\\
&\qquad \leq C_{p,\kappa,\epsilon} \mathbb{E}\bigg(\sup_{s\in[0,T]}|Y_s|^p\bigg) + \bigg(\int_{t\wedge\tau_n}^{\tau_n}f(s,0,0,0)ds\bigg)^p\\
& \qquad + \epsilon^\frac{p^2}{2}\bigg(\int_{t\wedge\tau_n}^{\tau_n}\int_{U} |V_s(e)|^p\lambda(de)ds\bigg)^\frac{p}{2}.
\end{split}
\end{eqnarray} 

By the equations (5.1) and (5.2) in \cite{Y}, we have

\begin{eqnarray}\label{Y2}
\mathbb{E}\bigg[\bigg(\int_{t\wedge\tau_n}^{\tau_n}\int_{U} |V_s(e)|^2 \mu(ds,de)\bigg)^\frac{p}{2}\bigg]\leq \mathbb{E}\bigg[ \int_{t\wedge\tau_n}^{\tau_n}\int_{U} |V_s(e)|^p  \lambda(de)ds\bigg]<\infty.
\end{eqnarray}

Thus choosing $\epsilon$ small enough we deduce that 

  \begin{eqnarray}   
 \begin{split}
&  \mathbb{E}\bigg[  \bigg( \int_{t\wedge\tau_n}^{\tau_n}|Z_s|^2 ds\bigg)^\frac{p}{2}+ \int_{t\wedge\tau_n}^{\tau_n}\int_{U} |V_s(e)|^p  \lambda(de)ds\bigg]\\
& \qquad \leq C_{p,\kappa,\epsilon}\mathbb{E}\bigg[\sup_{s\in[0,T]}|Y_s|^p+\bigg(\int_{t\wedge\tau_n}^{\tau_n}|f(s,0,0,0)|ds\bigg)^p \bigg].
\end{split}
\end{eqnarray}

Finally, letting $n$ to infinity and using Fatou's Lemma, (\ref{L1}) follows. $\square$\\

\begin{prop}\label{P} 
Assume that $(Y,Z,V)$ is a solution to BSDE (\ref{s2}) where $Y\in\mathcal{S}^p$, for some $p>1$. Then there exists a constant $C_{p,\kappa}$ such that ,
 $$\begin{array}{ll}
   \mathbb{E}\bigg[\sup_{t\in[0,T]} |Y_t|^p+ \bigg(\int_0^T|Z_s|^2ds\bigg)^{\frac{p}{2}}+\int_0^T\int_U |V_s(e)|^p\lambda(de)ds\bigg]   \leq C_{p,\kappa} \mathbb{E}\bigg[ |\xi|^p+\bigg(\int_0^T  |f(s,0,0,0)| ds\bigg)^p\bigg].
 
\end{array}$$
\end{prop}

\proof Applying It\^o 's formula to $|Y|^p$ over the interval $[t,T]$. Note that

$$ \begin{array}{ll}
 \frac{\partial\theta}{\partial y_i}(y)=py_i|y|^{p-2},\ \  \frac{\partial^2\theta}{\partial y_i\partial y_j}(y)=p|y|^{p-2}\delta_{i,j}+p(p-2)y_iy_j|y|^{p-4},
 \end{array}$$
where $\delta_{i,j}$ is the Kronecker delta. Thus, for evey $t\in[0,T]$, we have,

\begin{eqnarray}\label{I}
\begin{split}
 &  |Y_t|^p=|\xi|^p+p\int_t^T Y_s |Y_s|^{p-2} f(s,Y_s,Z_s,V_s)ds-p\int_t^T Y_s |Y_s|^{p-2} Z_s dB_s\\\\
 &  \qquad  -p\int_t^T \int_U Y_{s^-}|Y_{s^-}|^{p-2} V_s(e)\tilde{\mu}(ds,de)-\frac{1}{2}\int_t^T trace(D^2\theta(Y_s)Z_sZ^*_s)ds\\\\
 &  \qquad  -\int_t^T\int_U (|Y_{s^-}+V_s(e)|^p-|Y_{s^-}|^p-pY_{s^-}|Y_{s^-}|^{p-2}V_s(e))\mu(ds,de)   
   \end{split}
 \end{eqnarray}
 
First remark that for a non-negative symmetric matrix $\Gamma\in\mathbb{R}^{d\times d}$ we have

$$\begin{array}{ll}
\sum_{1\leq i,j\leq d} D^2\theta(y)_{i,j}\Gamma_{i,j}=p|y|^{p-2} trace(\Gamma)+p(p-2)|y|^{p-4}y^*\Gamma y\\
\qquad\qquad\qquad\qquad\geq p|y|^{p-2} trace(\Gamma),
\end{array}$$

then 
\begin{eqnarray}\label{I1}
 trace(D^2\theta(Y_s)Z_sZ^*_s)\geq p|y|^{p-2}|Z_s|^2.
\end{eqnarray}

Now for the Poisson quantity in (\ref{I}) following the same arguments as in (\cite{KP}, Prop 2) and (\cite{Y}, Lemma A.4), we obtain that

\begin{eqnarray}\label{I2}
\begin{split}
-\int_t^T\int_U (|Y_{s^-}+V_s(e)|^p-|Y_{s^-}|^p-pY_{s^-}|Y_{s^-}|^{p-2}V_s(e))\mu(ds,de)  \\
 \leq -p(p-1)3^{1-p}\int_t^T |Y_{s^-}|^{p-2}|V_s|^2\mu(ds,de).
\end{split}
\end{eqnarray}

Consequently, plugging (\ref{I1}) and (\ref{I2}) in the equation (\ref{I}) becomes

\begin{eqnarray}\label{I3}
\begin{split}
 &  |Y_t|^p+\frac{p(p-1)}{2}\int_t^T |Y_s|^{p-2} |Z_s|^2ds+\frac{p(p-1)}{2}\int_t^T \int_U |Y_s|^{p-2}|V_s|^2\mu(ds,de)\\\\ 
 & \leq |\xi|^p+p\int_t^T Y_s |Y_s|^{p-2} f(s,Y_s,Z_s,V_s)ds-p\int_t^T Y_s |Y_s|^{p-2} Z_s dB_s\\\\
 &  \qquad  -p\int_t^T \int_U Y_{s^-}|Y_{s^-}|^{p-2} V_s(e)\tilde{\mu}(ds,de).
   \end{split}
 \end{eqnarray}

 Since $f$ is Lipschitz then we have 
 
 $$\begin{array}{ll}
p Y_s f(s,Y_s,Z_s,V_s)\leq p|Y_s| f(s,0,0,0)+p\kappa|Y_s|^2+p\kappa |Y_s| |Z_s|+p\kappa |Y_s|  \Vert V_s\Vert_{\mathcal{L}^p}.
 \end{array}$$
 
 By Young's inequality (i.e. $ab\leq\frac{a^2}{2\epsilon}+\frac{\epsilon b^2}{2}$ for every $\epsilon>0$), we have
$$\begin{array}{ll}
     p\kappa| Y_s|   |Z_s|\leq \frac{p^2\kappa^2}{2\epsilon^2}|Y_s|^2+\frac{\epsilon^2}{2}  |Z_s|^2,
 
\end{array}$$

and by the inequality $ab\leq\frac{a^p}{p\epsilon^p}+\frac{\epsilon^q b^q}{q}$ for every $\epsilon>0$ with $\frac{1}{p}+\frac{1}{q}=1$, we get 
 
 $$\begin{array}{ll}
     p\kappa| Y_s| \Vert V_s\Vert_{\mathcal{L}^p}\leq  \frac{p-1}{p}(p\kappa)^\frac{p}{p-1}\epsilon^{-\frac{p}{p-1}}|Y_s|^\frac{p}{p-1} +\frac{\epsilon^p}{p}  \Vert V_s\Vert_{\mathcal{L}^p}^p.
 \end{array}$$

Therefore we have 

\begin{eqnarray} \label{Y3}
\begin{split}
 &  |Y_t|^p+(\frac{p(p-1)}{2}-\frac{\epsilon^2}{2})\int_t^T |Y_s|^{p-2} |Z_s|^2ds+\frac{p(p-1)}{2}\int_t^T \int_U |Y_s|^{p-2}|V_s|^2\mu(ds,de)\\\\ 
 & \leq |\xi|^p+p\int_t^T  |Y_s|^{p-1}f(s,0,0,0)ds+(p\kappa+\frac{p^2\kappa^2}{2\epsilon^2})\int_t^T |Y_s|^{p}ds\\
 & + \frac{p-1}{p}  (p\kappa)^\frac{p}{p-1}\epsilon^{-\frac{p}{p-1}}\int_t^T |Y_s|^{p-1}ds+\frac{\epsilon^p}{p}\int_t^T  \int_U |Y_s|^{p-2}|V_s(e)|^p\lambda(de)ds \\
 & -p \int_t^T Y_s |Y_s|^{p-2} Z_s dB_s -p\int_t^T \int_U Y_{s^-}|Y_{s^-}|^{p-2} V_s(e)\tilde{\mu}(ds,de).
   \end{split}
 \end{eqnarray}
 
Let us set $\Gamma_t=\int_0^t Y_s |Y_s|^{p-2} Z_s dB_s$ and $\Phi_t=\int_0^t  \int_U Y_{s^-}|Y_{s^-}|^{p-2} V_s(e)\tilde{\mu}(ds,de)$. Applying BDG's inequality we deduce that $\Gamma_t$ and $\Phi_t$ are uniformly integrable martingales. Indeed, by Young's inequality , we have
 
 \begin{eqnarray} 
\begin{split}
 & \mathbb{E}\bigg([\Gamma]_T^\frac{1}{2}\bigg)\leq  \mathbb{E}\bigg[ \sup_{s\in[0,T]}|Y_s|^{p-1}\bigg(\int_0^T |Z_s|^2 ds \bigg)^\frac{1}{2}\bigg]\\
& \qquad\qquad \leq\frac{p-1}{p}\mathbb{E}\bigg(\sup_{s\in[0,T]}|Y_s|^p\bigg)+\frac{1}{p}\mathbb{E}\bigg(\int_0^T |Z_s|^2 ds \bigg)^\frac{p}{2}<\infty,
   \end{split}
 \end{eqnarray}
 
 and from (\ref{Y2}),

 \begin{eqnarray} 
\begin{split}
 & \mathbb{E}\bigg([\Phi]_T^\frac{1}{2}\bigg)\leq  \mathbb{E}\bigg[ \sup_{s\in[0,T]}|Y_s|^{p-1}\bigg(\int_0^T |V_s(e)|^2 \mu(ds,de) \bigg)^\frac{1}{2}\bigg]\\
& \qquad\qquad \leq\frac{p-1}{p}\mathbb{E}\bigg(\sup_{s\in[0,T]}|Y_s|^p\bigg)+\frac{1}{p}\mathbb{E}\bigg(\int_0^T |V_s(e)|^2 \mu(ds,de) \bigg)^\frac{p}{2}\\
& \qquad\qquad \leq\frac{p-1}{p}\mathbb{E}\bigg(\sup_{s\in[0,T]}|Y_s|^p\bigg)+\frac{1}{p} \mathbb{E}\bigg[ \int_{t\wedge\tau_n}^{\tau_n}\int_{U} |V_s(e)|^p  \lambda(de)ds\bigg]<\infty.
   \end{split}
 \end{eqnarray}
 
 Thus taking expectation in (\ref{Y3}) leads to that,
 
 \begin{eqnarray} \label{Y4}
\begin{split}
 &(\frac{p(p-1)}{2}-\frac{\epsilon^2}{2}) \mathbb{E}\int_t^T |Y_s|^{p-2} |Z_s|^2ds+\frac{p(p-1)}{2}\mathbb{E}\int_t^T \int_U |Y_s|^{p-2}|V_s(e)|^2 \lambda(de)ds\\\\ 
 & \leq\mathbb{E}\bigg[ |\xi|^p+p\int_t^T  |Y_s|^{p-1}f(s,0,0,0)ds+(p\kappa+\frac{p^2\kappa^2}{2\epsilon^2})\int_t^T |Y_s|^{p}ds \\
 & + \frac{p-1}{p}  (p\kappa)^\frac{p}{p-1}\epsilon^{-\frac{p}{p-1}}\int_t^T |Y_s|^{p-1}ds+\frac{\epsilon^p}{p}\int_t^T  \int_U |Y_s|^{p-2}|V_s(e)|^p\lambda(de)ds \bigg].
   \end{split}
 \end{eqnarray}

 Taking account of (\ref{Y4}), (\ref{Y3}) becomes
 
 \begin{eqnarray} \label{Y5}
\begin{split}
&  \mathbb{E}\sup_t|Y_t|^p \leq \mathbb{E}\bigg[ |\xi|^p+p\int_t^T  |Y_s|^{p-1}f(s,0,0,0)ds+(p\kappa+\frac{p^2\kappa^2}{2\epsilon^2})\int_t^T |Y_s|^{p}ds \\
 & + \frac{p-1}{p}  (p\kappa)^\frac{p}{p-1}\epsilon^{-\frac{p}{p-1}}\int_t^T |Y_s|^{p-1}ds+\frac{\epsilon^p}{p}\int_t^T  \int_U |Y_s|^{p-2}|V_s(e)|^p\lambda(de)ds \bigg]\\
 &   +p\mathbb{E}\bigg(\sup_{s\in[0,T]}\bigg|\int_s^T Y_{s}|Y_{s}|^{p-2} Z_s dB_s\bigg|\bigg)+p\mathbb{E}\bigg(\sup_{s\in[0,T]}\bigg|\int_s^T \int_U  Y_{s^-}|Y_{s^-}|^{p-2}V_s(e) \tilde{\mu}(ds,de)\bigg|\bigg).
    \end{split}
\end{eqnarray}

Using the BDG inequality and Young's inequality (i.e. $ab\leq \frac{a^p}{p}+\frac{b^q}{q}$, with $\frac{1}{p}+\frac{1}{q}=1$) we get

$$\begin{array}{ll}
 & p\mathbb{E}\bigg[\sup_{s\in[0,T]}\bigg|\int_s^T Y_{s}|Y_{s}|^{p-2} Z_s dB_s  \bigg|\bigg] \leq C_p \mathbb{E}\bigg[\bigg(\int_t^T |Y_s|^{2(p-1)} |Z_s|^2ds   \bigg)^{1/2}\bigg]\\
 &\qquad\qquad   \leq C_p \mathbb{E}\bigg[\bigg(\sup_{s\in[0,T]} |Y_s|^{p/2}\bigg) \bigg(\int_t^T |Y_s|^{p-2} |Z_s|^2ds \bigg)^{1/2}\bigg]\\
 &\qquad\qquad  \leq \frac{C_p}{2} \mathbb{E}\bigg[\sup_{s\in[0,T]} |Y_s|^p\bigg]+\frac{1}{2}\mathbb{E}\bigg[ \int_t^T |Y_s|^{p-2} |Z_s|^2ds \bigg].
 
\end{array}$$

and, 

$$\begin{array}{ll}
&  p\mathbb{E}\bigg[ \sup_{s\in[0,T]}\bigg|\int_s^T \int_U  Y_{s^-}|Y_{s^-}|^{p-2}V_s(e)\tilde{\mu}(ds,de)\bigg|\bigg]\\
&\qquad  \leq C_p \mathbb{E}\bigg[\bigg( \int_t^T \int_U|Y_{s}|^{p(p-2)} |V_s(e)|^p\mu(ds,de)\bigg)^{1/p}\bigg]\\
&\qquad   \leq C_p \mathbb{E}\bigg[\bigg(\sup_{s\in[0,T]} |Y_s|^{ \frac{(p-2)(p-1)}{p}}\bigg) \bigg( \int_t^T \int_U|Y_{s}|^{p-2} |V_s(e)|^p\mu(ds,de)\bigg)^{1/p}\bigg]\\
&\qquad   \leq \frac{ p-1}{p} C_p \mathbb{E}\bigg[\sup_{s\in[0,T]} |Y_s|^{p-2}\bigg]+\frac{C_p}{p}  \mathbb{E}\bigg[  \int_t^T \int_U |Y_{s}|^{p-2}| V_s|^p  \mu(ds,de) \bigg].
 
\end{array}$$

Coming back to inequality (\ref{Y5}) with the above estimates we deduce that

 \begin{eqnarray}  
\begin{split}
&  \mathbb{E}\sup_t|Y_t|^p \leq  \mathbb{E}\bigg[ |\xi|^p+p\int_t^T  |Y_s|^{p-1}f(s,0,0,0)ds+(p\kappa+\frac{p^2\kappa^2}{2\epsilon^2})\int_t^T |Y_s|^{p}ds \\
 & + \frac{p-1}{p}  (p\kappa)^\frac{p}{p-1}\epsilon^{-\frac{p}{p-1}}\int_t^T |Y_s|^{p-1}ds  \bigg] 
    \end{split}
\end{eqnarray}

Applying once again Young's inequality, we get

$$\begin{array}{ll}
   & p \int_t^T |Y_s|^{p-1}|f(s,0,0,0)|ds\leq p C_p \bigg(\sup_{s\in[0,T]}|Y_s|^{p-1}\int_t^T |f(s,0,0,0)| ds\bigg)\\
    
    & \qquad\qquad\leq  C_p \bigg(\sup_{s\in[0,T]}|Y_s|^p\bigg)+\frac{1}{p}\bigg(\int_t^T |f(s,0,0,0)|ds\bigg)^p,
 
\end{array}$$

where $C_p$ changes from a line to another. Consequently, 

$$\begin{array}{ll}
   \mathbb{E}\sup_{t\in[0,T]}|Y_t|^p \leq C'_{p} \mathbb{E}\bigg[ |\xi|^p+\bigg(\int_t^T  |f(s,0,0,0)|ds\bigg)^p\bigg]+C''_{p,\kappa} \int_t^T \mathbb{E}\sup_{u\in[s,T]}|Y_u|^pds.
\end{array}$$

Finally, using Gronwall's lemma, ve obtain 

$$\begin{array}{ll}
   \mathbb{E}\sup_{t\in[0,T]}|Y_t|^p \leq C'_{p} e^{C''_{p,\kappa} T} \mathbb{E}\bigg[ |\xi|^p+\bigg(\int_0^T  |f(s,0,0,0)|ds\bigg)^p\bigg].
\end{array}$$

The desired result follows from Lemma \ref{L}. $\square$

\subsection{Uniqueness}
\begin{lem} Under assumptions (A1) and (A2) on $(f,\xi)$, the associated BSDE has at most one solution $(Y,Z,V)$ such that $Y$ belongs to the class $\mathbb{D}$, $Z\in\cup_{\beta>\alpha}\mathcal{M}^\beta$ and $V\in\mathcal{L}^1$.
\end{lem}

\begin{proof} Assume that $(Y,Z,V)$ and $(Y',Z',V')$ are two solutions of (\ref{s2}). For any $n\geq 0$, let us define $\tau_n$ as follows:
$$\tau_n=\inf\{t\geq 0;\int_0^t(|Z_s|^2+|Z'_s|^2)ds+\int_t^T\int_{\mathcal{U}}(|V_s|^2+|V'_s|^2)  \lambda(de)ds>n\}\wedge T.$$

We first show that there exists a constant $p>1$ such that $Y-Y'$ belongs to $\mathcal{S}^p$. Applying It\^o-Tanaka's formula on $[t\wedge\tau_n,\tau_n]$ gives,

$$\begin{array}{ll}
\begin{split}
& |Y_{t\wedge\tau_n}-Y'_{t\wedge\tau_n}|=|Y_{ \tau_n}-Y'_{ \tau_n}|+\int_{t\wedge\tau_n}^{\tau_n} sgn(Y_{s^-}-Y'_{s^-}) d(Y_s-Y'_s)+ L_t^0(Y-Y')
 ,
\end{split}
\end{array}$$

where the process $(L_t^0(Y-Y'))_{t\leq T}$ is the local time of the semi martingale $(Y_s-Y'_s)_{0\leq s\leq T}$ at $0$ which is a non-negative process (for more details one can see theorem 68 in \cite{PR} p.213) and $sgn(Y-Y'):=\frac{Y-Y'}{|Y-Y'|}\mathbf{1}_{Y\neq Y'}$. Then we have,

\begin{eqnarray}\label{e}
\begin{split}
& |Y_{t\wedge\tau_n}-Y'_{t\wedge\tau_n}|\leq|Y_{ \tau_n}-Y'_{ \tau_n}| +\int_{t\wedge\tau_n}^{\tau_n} sgn(Y_{s^-}-Y'_{s^-})
\bigg[f(s,Y_s,Z_s,V_s)-f(s,Y'_s,Z'_s,V'_s)\bigg]ds\\
& \qquad \qquad \qquad +\int_{t\wedge\tau_n}^{\tau_n} sgn(Y_{s^-}-Y'_{s^-})  (Z_s-Z'_s) dB_s\\
& \qquad \qquad \qquad -\int_{t\wedge\tau_n}^{\tau_n} \int_{U}sgn(Y_{s^-}-Y'_{s^-})(V_s(e)-V'_s(e))\tilde{\mu}(ds,de) .
\end{split}
\end{eqnarray}
 
Thus, using the (iv) assumption on $f$, we get

$$\begin{array}{ll}
\begin{split}
& sgn(Y_{s^-}-Y'_{s^-})
\bigg[f(s,Y_s,Z_s,V_s)-f(s,Y'_s,Z'_s,V'_s)\bigg]=sgn(Y_{s^-}-Y'_{s^-})\bigg[f(s,Y_s,Z_s,V_s)-f(s,Y_s,0,0)\\
& \qquad \qquad \qquad +f(s,Y_s,0,0)-f(s,Y'_s,0,0)+f(s,Y'_s,0,0)-f(s,Y'_s,Z'_s,V'_s)\bigg]\\
& \qquad \qquad \qquad \leq C sgn(Y_{s^-}-Y'_{s^-})\bigg(g_s+|Y_s|+|Y'_s|+|Z_s|+|Z'_s|+\Vert V_s \Vert+\Vert V'_s \Vert\bigg)^\alpha.
 \end{split}
\end{array}$$

Next taking conditional expectation in (\ref{e}) with respect to $\mathcal{F}_{t\wedge\tau_n}$ on both hand-sides and taking into account that the two last terms are $\mathcal{F}_{t\wedge\tau_n}$- martingales we obtain,

 $$\begin{array}{ll}
\begin{split}
& \bigg|Y_{t\wedge\tau_n}-Y'_{t\wedge\tau_n}\bigg|\leq \mathbb{E}\bigg[|Y_{ \tau_n}-Y'_{ \tau_n}| +C \int_{t\wedge\tau_n}^{\tau_n} sgn(Y_{s^-}-Y'_{s^-})\\
& \qquad \qquad \bigg(g_s+|Y_s|+|Y'_s|+|Z_s|+|Z'_s|+\Vert V_s \Vert+\Vert V'_s \Vert\bigg)^\alpha ds|\mathcal{F}_{t\wedge\tau_n}\bigg]. 
\end{split}
\end{array}$$
 
 Taking now the limit as $n\rightarrow\infty$ we get
 
 \begin{eqnarray}\label{i4}
 \bigg|Y_{t}-Y'_{t}\bigg|\leq C \mathbb{E}\bigg[ \int_{t}^{T} sgn(Y_{s^-}-Y'_{s^-})\bigg(g_s+|Y_s|+|Y'_s|+|Z_s|+|Z'_s|+\Vert V_s \Vert+\Vert V'_s \Vert\bigg)^\alpha ds|\mathcal{F}_{t }\bigg]. 
 \end{eqnarray}
 
For $p>1$ such that $p\alpha<\beta$, applying Doob's inequality to have

   $$\begin{array}{ll}
 
\mathbb{E}\bigg[\sup_{t\leq T}|Y_{t}-Y'_{t}|^p\bigg]\leq C \mathbb{E}\bigg[ \int_{0}^{T} sgn(Y_{s^-}-Y'_{s^-})\bigg(g_s+|Y_s|+|Y'_s|+|Z_s|+|Z'_s|+\Vert V_s \Vert+\Vert V'_s \Vert\bigg)^{\alpha p} ds\bigg]<\infty. 
 
\end{array}$$
 
Hence  $|Y-Y'|$ belongs to $\mathcal{S}^p$, for some $p>1$.\\

Now let $p>1$. By It\^o-Tanaka formula on $[t\wedge\tau_n,\tau_n]$ we have

\begin{eqnarray}\label{e}
\begin{split}
& |Y_{t\wedge\tau_n}-Y'_{t\wedge\tau_n}|^p+\frac{p(p-1)}{2}\int_{t\wedge\tau_n}^{\tau_n}|Y_s-Y'_s|^{p-2}sgn(Y_{s^-}-Y'_{s^-})  |Z_s-Z'_s|^2ds\\
&\qquad +\frac{p(p-1)}{2}\int_{t\wedge\tau_n}^{\tau_n} \int_{U} |Y_s-Y'_s|^{p-2}sgn(Y_{s^-}-Y'_{s^-})  (V_s(e)-V'_s(e))^2 \mu(ds,de)  \\ 
& \qquad =|Y_{ \tau_n}-Y'_{ \tau_n}|^p +p\int_{t\wedge\tau_n}^{\tau_n}
sgn(Y_{s^-}-Y'_{s^-})|Y_s-Y'_s|^{p-1}\bigg[f(s,Y_s,Z_s,V_s)-f(s,Y'_s,Z'_s,V'_s)\bigg]ds\\
&\qquad -p\int_{t\wedge\tau_n}^{\tau_n} sgn(Y_{s^-}-Y'_{s^-})|Y_s-Y'_s|^{p-1}   (Z_s-Z'_s) dB_s\\
& \qquad -p \int_{t\wedge\tau_n}^{\tau_n} \int_{U} sgn(Y_{s^-}-Y'_{s^-})|Y_s-Y'_s|^{p-1} (V_s(e)-V'_s(e))^2\tilde{\mu}(ds,de)  .
\end{split}
\end{eqnarray}
 From the Lipschitz property of $f$, there exist bounded and $\mathcal{F}_{t }$-adapted processes $(a_t)_{t\in[0,T]}$, $(b_t)_{t\in[0,T]}$ and $(c_t)_{t\in[0,T]}$ such that 
 
$$\begin{array}{ll}
 f(s,Y_s,Z_s,V_s)-f(s,Y'_s,Z'_s,V'_s)=a_s(Y_s-Y'_s)+b_s(Z_s-Z'_s)+c_s\int_{U}(V_s(e)-V'_s(e)) \lambda(de).

\end{array}$$

Therefore for any $t\leq T$, the equality (\ref{e}) becomes

\begin{eqnarray}  \label{e1}
\begin{split}
& |Y_{t\wedge\tau_n}-Y'_{t\wedge\tau_n}|^p+\frac{p(p-1)}{2}\int_{t\wedge\tau_n}^{\tau_n}|Y_s-Y'_s|^{p-2}sgn(Y_{s^-}-Y'_{s^-})  |Z_s-Z'_s|^2ds\\
&\qquad +\frac{p(p-1)}{2}\int_{t\wedge\tau_n}^{\tau_n} \int_{U} |Y_s-Y'_s|^{p-2}sgn(Y_{s^-}-Y'_{s^-})  (V_s(e)-V'_s(e))^2 \mu(ds,de)  \\ 
& \qquad \leq |Y_{ \tau_n}-Y'_{ \tau_n}|^p +p\kappa\int_{t\wedge\tau_n}^{\tau_n}
sgn(Y_{s^-}-Y'_{s^-})|Y_s-Y'_s|^{p} ds\\
&\qquad +p\kappa\int_{t\wedge\tau_n}^{\tau_n}
sgn(Y_{s^-}-Y'_{s^-})|Y_s-Y'_s|^{p-1}|Z_s-Z'_s|ds\\
&\qquad +p\kappa\int_{t\wedge\tau_n}^{\tau_n}\int_{U}
sgn(Y_{s^-}-Y'_{s^-})|Y_s-Y'_s|^{p-1}(V_s(e)-V'_s(e))\lambda( de)ds\\
&\qquad -p\int_{t\wedge\tau_n}^{\tau_n} sgn(Y_{s^-}-Y'_{s^-})|Y_s-Y'_s|^{p-1}   (Z_s-Z'_s) dB_s\\
& \qquad -p \int_{t\wedge\tau_n}^{\tau_n} \int_{U} sgn(Y_{s^-}-Y'_{s^-})|Y_s-Y'_s|^{p-1} (V_s(e)-V'_s(e))^2\tilde{\mu}(ds,de)  .
\end{split}
\end{eqnarray}

Applying Young's inequality (i.e. $ab\leq\frac{a^2}{2\epsilon}+\frac{\epsilon b^2}{2}$ with $\epsilon=\frac{p-1}{2p}$)

$$\begin{array}{ll}
 p\kappa |Y_s-Y'_s|^{p-1}|Z_s-Z'_s|\leq \frac{p\kappa^2}{(p-1)}|Y_s-Y'_s|^p+\frac{p(p-1)}{4}|Y_s-Y'_s|^{p-2}|Z_s-Z'_s|^2,
 
\end{array}$$

and by $ab\leq\frac{a^p}{p\epsilon^p}+\frac{\epsilon^q b^q}{q}$ with $\frac{1}{p}+\frac{1}{q}=1$ we have ,

$$\begin{array}{ll}
 p\kappa |Y_s-Y'_s|^{p-1}(V_s(e)-V'_s(e)) \leq (p-1)\kappa^\frac{p}{(p-1)}p^{p-1}\epsilon^{-\frac{p}{p-1}}|Y_s-Y'_s|^{p-1}+\frac{\epsilon^p}{p}|Y_s-Y'_s|^{p-2}(V_s(e)-V'_s(e))^p.
 
\end{array}$$

Following the same arguments as in the proof of proposition \ref{P} and going back to (\ref{e1}) to obtain
$$\begin{array} {ll}
\begin{split}
& |Y_{t\wedge\tau_n}-Y'_{t\wedge\tau_n}|^p  \leq |Y_{ \tau_n}-Y'_{ \tau_n}|^p +(p\kappa+\frac{p\kappa^2}{(p-1)})\int_{t\wedge\tau_n}^{\tau_n}
sgn(Y_{s^-}-Y'_{s^-})|Y_s-Y'_s|^{p} ds\\
&\qquad +(p-1)\kappa^\frac{p}{(p-1)}p^{p-1}\epsilon^{-\frac{p}{p-1}}\int_{t\wedge\tau_n}^{\tau_n}
sgn(Y_{s^-}-Y'_{s^-})|Y_s-Y'_s|^{p-1}ds\\
& \qquad -p\int_{t\wedge\tau_n}^{\tau_n} sgn(Y_{s^-}-Y'_{s^-})|Y_s-Y'_s|^{p-1}   (Z_s-Z'_s) dB_s\\
&\qquad -p \int_{t\wedge\tau_n}^{\tau_n}\int_{U} sgn(Y_{s^-}-Y'_{s^-})|Y_s-Y'_s|^{p-1} (V_s(e)-V'_s(e))^2\tilde{\mu}(ds,de)  .
\end{split}
\end{array}$$

Finally taking expectation and since the two latter terms are martingales due to Lemma \ref{Y1}, we have

$$\begin{array}{ll}
 \mathbb{E}\bigg[|Y_{t\wedge\tau_n}-Y'_{t\wedge\tau_n}|^p\bigg]\leq\mathbb{E}\bigg[|Y_{ \tau_n}-Y'_{ \tau_n}|^p +C_{p,\kappa}\int_{t\wedge\tau_n}^{\tau_n}
sgn(Y_{s^-}-Y'_{s^-})|Y_s-Y'_s|^{p} ds\bigg]
\end{array}$$

As for some $p>1$, $|Y-Y'|\in\mathcal{S}^p$ then taking the limit with respect to $n$ to get

$$\begin{array}{ll}
 \mathbb{E}\bigg[|Y_{t }-Y'_{t }|^p\bigg]\leq C_{p,\kappa}  \mathbb{E}\bigg[ \int_{t }^{T}
sgn(Y_{s^-}-Y'_{s^-})|Y_s-Y'_s|^{p} ds\bigg].
\end{array}$$

From Gronwall's lemma we conclude that $\mathbb{E}[|Y_{t }-Y'_{t }|^p]=0$, $\forall t\leq T$. Then $Y_t=Y'_t$, for all $t\leq T$.\\

Since there exist $\beta>\alpha$ and $\beta'>\alpha$ such that $Z\in\mathcal{M}^{\beta}$ and $Z'\in\mathcal{M}^{\beta'}$, then $Z$ and $Z'$ belong to $\mathcal{M}^{\beta\vee\beta'}$. We have also that $V=V'$. Consequently $(Y,Z,V)=(Y',Z',V')$.  We conclude that the BSDE (\ref{s2}) has at most one solution $(Y,Z,V)$ such that $Y$ belongs to the class $\mathbb{D}$, $Z\in\cup_{\beta>\alpha}\mathcal{M}^\beta$ and $V\in\mathcal{L}^1$.
 $\square$\\
\end{proof}

\subsection{Existence of a $\mathbb{D}$-solution}

We will need the following assumption on the data, for some $p>1$,\\

(A3) $$\mathbb{E}[|\xi|^p+(\int_0^T|f(s,0,0,0)|ds)^p]<+\infty.$$

Let us recall the following result. A proof can be found in (\cite{LT}, Lemma 2.4).

\begin{thm}\label{T1} For $p=2$, under assumptions (A2ii) and (A3) on the data $(\xi,f)$ then there exists a unique triple $(Y,Z,V)\in\mathcal{S}^2\times\mathcal{M}^2\times\mathbb{L}^2(\tilde{\mu})$ which solves the following BSDE

$$\begin{array}{ll}
   Y_t=\xi+\int_{t}^{T}f(s,Y_s,Z_s,V_s)ds -\int_{t}^{T}Z_s dB_s-\int_t^T \int_U V_s(e) \tilde{\mu}(ds,de),\ 0\leq t\leq T.
\end{array}$$ 
\end{thm}

For a given $p\in(1,2)$, (\cite{Y}, Theorem 4.1) the author proved the following result. 

\begin{thm}\label{T2} For $p\in(1,2)$, under assumptions (A2ii) and (A3) on the data $(\xi,f)$ then the BSDE with jumps admits a unique solution $(Y,Z,V)\in\mathcal{S}^p\times\mathcal{M}^p\times\mathcal{L}^p$.  

$$\begin{array}{ll}
   Y_t=\xi+\int_{t}^{T}f(s,Y_s,Z_s,V_s)ds -\int_{t}^{T}Z_s dB_s-\int_t^T \int_U V_s(e) \tilde{\mu}(ds,de),\ 0\leq t\leq T.
\end{array}$$ 
\end{thm}
\vspace{0.8cm}
We now prove our main existence result for $p=1$.

\begin{thm}\label{T}
Let assumptions (A1) and (A2) on $(f,\xi)$ hold. Then the associated BSDE (\ref{s2}) has a solution $(Y,Z,V)$ such that $Y$ belongs to class $\mathbb{D}$ and, for each $\beta\in(0,1)$, $(Z,V)\in\mathcal{M}^\beta\times\mathcal{L}^1$.
\end{thm}

Before giving the proof of this result, we study the case where the generator is independent of the variables $z$ and $v$.\\

\begin{prop} Let assumptions (A1) and (A2) on $(f,\xi)$ hold and let us suppose that $f$ does not depend on $z$ and $v$. Then the associated BSDE (\ref{s2}) has a solution $(Y,Z,V)$ such that $Y$ belongs to class $\mathbb{D}$ and, for each $\beta\in(0,1)$, $(Z,V)\in\mathcal{M}^\beta\times\mathcal{L}^1$.
\end{prop}

\proof Let us set for each integer $n\geq 1$, $\xi_n=q_n(\xi)$ and $f_n(t,y)=f(t,y)-f(t,0)+q_n(f(t,0))$ as in the proof of Theorem \ref{T2}. It follows from this result that the BSDE associated to the couple $(\xi_n,f_n)$ has a unique solution in $\mathcal{S}^2\times\mathcal{M}^2\times\mathcal{L}^2 $.\\

Using It\^o-Tanaka formula as in the proof of the uniqueness result we have 

 $$\begin{array}{ll}
 & |Y_{t}^{n+i}-Y^n_{t}|\leq   |\xi_{n+i}-\xi_n|+ \int_{t}^{T}  (f_{n+i}(s,Y_s^{n+i})-f_n(s,Y_s^n))ds+ \int_{t}^{T}  (Z_{s}^{n+i}-Z^n_{s})dB_s\\\\
 
 & -\int_{t }^{T} \int_{\mathcal{U}} (V_s^{n+i}(e)-V^{n}_s(e))\tilde{\mu}(ds,de).
 \end{array}$$

Now taking conditional expectation with respect to $\mathcal{F}_{t}$ on both hand-sides and taking into account that the two last terms are $\mathcal{F}_{t }$- martingales we obtain, 

 $$\begin{array}{ll}
 & |Y_{t}^{n+i}-Y^n_{t}|\leq \mathbb{E}[  |\xi_{n+i}-\xi_n|+ \int_{t}^{T} |f_{n+i}(s,Y_s^{n+i})-f_n(s,Y_s^n)|ds|\mathcal{F}_t]. 
 
 \end{array}$$
 
 We deduce that
 
 $$\begin{array}{ll}
 |Y_{t}^{n+i}-Y^n_{t}|\leq  \mathbb{E}[|\xi|\mathbf{1}_{\xi>n}+ \int_{t}^{T}  |f(s,0)|\mathbf{1}_{f(s,0)>n}ds|\mathcal{F}_{t }]. 
 \end{array}$$

Therefore 

 $$\begin{array}{ll}
 \Vert Y^{n+i}-Y^n\Vert_{\mathbb{D}}\leq  \mathbb{E}[|\xi|\mathbf{1}_{\xi>n}+ \int_{t}^{T}   |f(s,0)|\mathbf{1}_{f(s,0)>n}ds]. 
 \end{array}$$

So $(Y^n)$ is a Cauchy sequence for the norm $\Vert.\Vert_\mathbb{D}$ converges to the progressive measurable process limit $Y$ which belongs to the class $\mathbb{D}$.\\

Let $(Y^{n+i}-Y^n,Z^{n+i}-Z^n,V^{n+i}-V^n)$ be the solution of the following BSDE:

 $$\begin{array}{ll}
  Y^{n+i}_t-Y^n_t=\xi_{n+i}-\xi_n+\int_t^T (f_{n+i}(s,Y_s^{n+i})-f_n(s,Y_s^n))ds-\int_t^t (Z^{n+i}_s-Z^n_s) dB_s\\
  \qquad \qquad -\int_t^T \int_U (V^{n+i}_s(e)-V^n_s(e)) \tilde{\mu}(ds,de).
 \end{array}$$

The random function $f_{n+i}(s,Y_s^{n+i})-f_n(s,Y_s^n)$ verifies the Lipschitz property as $f$ then by Lemma \ref{L} we deduce that for $\beta\in(0,1)$,

$$\begin{array}{ll}
   \mathbb{E}[(\int_0^T(|Z^{n+i}_s-Z^n_s|^2+\Vert V^{n+i}_s-V^n_s\Vert^2) ds)^{\frac{\beta}{2}}]\leq C_{\beta,\kappa} \mathbb{E}[\sup_{t}|Y^{n+i}_t-Y^n_t|^\beta+(\int_0^T |f(s,0)| \mathbf{1}_{f(s,0)>n}ds)^\beta].

\end{array}$$

 Hence both $(Z^n)$ and $(V^n)$ are Cauchy sequences, for each $\beta\in(0,1)$, in the spaces $\mathcal{M}^\beta$ and $\mathcal{L}$ which converge to measurable processes $Z$ and $V$.\\
 
 So we get that $(Y^n,Z^n,V^n)$ solution of the following BSDE
 
 \begin{eqnarray}
  Y_t^n=\xi_n+\int_{t}^{T}f_n(s,Y_s^n)ds -\int_{t}^{T}Z_s^n dB_s-\int_t^T \int_U V_s^n(e) \tilde{\mu}(ds,de), t\leq T;
 \end{eqnarray}
 
 Since $\int_{0}^{t}Z_s^n dB_s$ coverges to $\int_{0}^{t}Z_s dB_s$, also $\int_0^t \int_U V_s^n(e) \tilde{\mu}(ds,de)$ converges to $\int_0^t \int_U V_s(e) \tilde{\mu}(ds,de)$ and since the map $y\mapsto f(t,y)$ is continuous, by taking the limit we check easily that the limit $(Y,Z,V)$ solves the following desired BSDE
  $$\begin{array}{ll}
  Y_t=\xi+\int_{t}^{T}f(s,Y_s)ds -\int_{t}^{T}Z_s dB_s-\int_t^T \int_U V_s(e) \tilde{\mu}(ds,de). \ \square
 \end{array} $$

Now we can prove our main existence result.\\

 \proof \textbf{of Theorem \ref{T}}  We will finally complete the proof of the existence part. To this end, we consider a Picard's iteration. Let us set $(Y^0,Z^0,V^0)=(0,0,0)$ and define recursively, for each $n\geq 0$,
 
$$\begin{array}{ll}
   Y_t^{n+1}=\xi+\int_{t}^{T}f(s,Y_s^{n+1},Z_s^n,V_s^n)ds -\int_{t}^{T}Z_s^{n+1} dB_s-\int_t^T \int_U V_s^{n+1}(e) \tilde{\mu}(ds,de), 0\leq t\leq T.
\end{array}$$ 

For $n\geq 1$, following the same arguments as in the proof of uniqness, we obtain that

 $$\begin{array}{ll}
 |Y_t^{n+1}-Y_t^n|\leq C \mathbb{E}[ \int_{t}^{T} sgn(Y_{s^-}^n-Y_{s^-}^{n-1})(g_s+|Y_s^n|+|Y_s^{n-1}|+|Z_s^n|+|Z_s^{n-1}|+\Vert V_s^n \Vert+\Vert V_s^{n-1} \Vert)^\alpha ds|\mathcal{F}_{t }]. 
 
\end{array}$$

$Z^n$ and $Z^{n-1}$ belong to $\mathcal{M}^\beta$ for each $\beta\in(0,1]$, $Y^n$ and $Y^{n-1}$ belong to class $\mathbb{D}$, $V^n$ and $V^{n-1}$ belong to $\mathcal{L}$ and $(g_t)_{t\in[0,T]}$ is integrable. Hence the quantity $$\int_{t}^{T} sgn(Y_{s^-}^n-Y_{s^-}^{n-1})(g_s+|Y_s^n|+|Y_s^{n-1}|+|Z_s^n|+|Z_s^{n-1}|+\Vert V_s^n \Vert+\Vert V_s^{n-1} \Vert)^\alpha ds$$

belongs to the space $\mathbb{L}^q$ such that $\alpha q<1$. Let us fix $q\in(1,2)$ such that $\alpha p<1$. Then for all $n\geq 1$, $y^n=Y^{n+1}-Y^n$ belongs to $\mathcal{S}^q$. Let us set $z^n=Z^{n+1}-Z^n$ and $v^n=V^{n+1}-V^n$. The triple $(y^n,z^n,v^n)$ is a solution of the following BSDE 

 $$\begin{array}{ll}
   y_t^n= \int_{t}^{T}f_n(s,y_s^n )ds -\int_{t}^{T}z_s^n dB_s-\int_t^T \int_U v_s^n(e) \tilde{\mu}(ds,de), t\leq T,
 
\end{array}$$

where the generator $f_n(s,y_s^n)=f(s,Y^{n+1},Z_s^n,V_s^n)-f(s,Y_s^n,Z_s^{n-1},V_s^{n-1})$. Since $f$ assumed to satisfy (A2), $f_n$ verifies it too.\\

From Lemma \ref{L} we have that $z^n$ belongs to $\mathcal{M}^q$ since $y^n$ is in $\mathcal{S}^q$ and by Proposition \ref{P} we obtain that there exists a constant $C_{q,\kappa}$ such that 
 
 $$\begin{array}{ll}
   \mathbb{E}[\sup_t |y_t^n|^q+(\int_0^T (|z_s^n|^2+\Vert v_s^n\Vert^2 )ds)^{\frac{q}{2}}]\leq C_{q,\kappa} \mathbb{E}[ (\int_0^T  |f(s,Y_s^n,Z_s^n,V_s^n)-f(s,Y_s^n,Z_s^{n-1},V_s^{n-1})| ds)^q].

\end{array}$$

For $n\geq 2$, by the Lipschitz property on $f$, we get
 
  $$\begin{array}{ll}
   \mathbb{E}[\sup_t |y_t^n|^q+(\int_0^T (|z_s^n|^2+\Vert v_s^n\Vert^2 )ds)^{\frac{q}{2}}]\leq \kappa\ C_{q,\kappa}  \mathbb{E}[ (\int_0^T  |z_s^{n-1}|+\Vert v_s^{n-1}\Vert ds)^q],
 
\end{array}$$

applying H\"{o}lder's inequality, we obtain

$$\begin{array}{ll}
   \mathbb{E}[\sup_t |y_t^n|^q+(\int_0^T (|z_s^n|^2+\Vert v_s^n\Vert^2 )ds)^{\frac{q}{2}}]\leq \kappa\ C_{q,\kappa} T^{1-\frac{q}{2}} \mathbb{E}[ (\int_0^T  (|z_s^{n-1}|+\Vert v_s^{n-1}\Vert)^2 ds)^{\frac{q}{2}}].
 
\end{array}$$

Hence for  $n\geq 2$, we have

$$\begin{array}{ll}
   \mathbb{E}[\sup_t |y_t^n|^q+(\int_0^T (|z_s^n|^2+\Vert v_s^n\Vert^2 )ds)^{\frac{q}{2}}]\leq (\kappa\ C_{q,\kappa} T^{1-\frac{q}{2}})^{n-1} \mathbb{E}[\sup_t |y_t^1|^q+(\int_0^T (|z_s^1|^2+\Vert v_s^1\Vert^2 )ds)^{\frac{q}{2}}].
 
\end{array}$$

Let us first assume, for a sufficiently small $T$, that $\kappa\ C_{q,\kappa} T^{1-\frac{q}{2}}<1$. Then the term of the right-hand side of the last inequality is finite, we deduce that $(Y^n-Y^1,Z^n-Z^1,V^n-V^1)$ converges to $(U,V,W)$ in the space $\mathcal{S}^q\times\mathcal{M}^q\times\mathcal{L}$ therefore the quantity $\lbrace\mathbb{E}[\sup_t |y_t^1|^q+(\int_0^T (|z_s^1|^2+\Vert v_s^1\Vert^2 )ds)^{\frac{q}{2}}]\rbrace$ is finite.\\

 Therefore $(Y^n,Z^n,V^n)$ converges to $(Y=U+Y^1,Z=V+Z^1,V=W+V^1)$ in the space $\mathcal{S}^\beta\times\mathcal{M}^\beta\times\mathcal{L}$ for each $\beta\in(0,1]$ since $(Y^1,Z^1,V^1)$ belongs to it. Also we deduce that $Y^n$ converges to $Y$ for the norm $\Vert.\Vert_{\mathbb{D}}$ since $Y^1$ belongs to class $\mathbb{D}$ and the convergence in $\mathcal{S}^q$ with $q>1$ is stronger than the convergence in $\Vert.\Vert_{\mathbb{D}}$-norm.\\
 
 We conclude, by taking the limit in the following equtaion satisfied by $(Y^n,Z^n,V^n)$ as follows,
 
 $$\begin{array}{ll}
   Y_t^{n}=\xi+\int_{t}^{T}f(s,Y_s^{n},Z_s^{n-1},V_s^{n-1})ds -\int_{t}^{T}Z_s^{n} dB_s-\int_t^T \int_U V_s^{n}(e) \tilde{\mu}(ds,de),  
\end{array}$$ 

that the triple $(Y,Z,V)$ which belongs to the space $ \mathbb{D}\times\mathcal{M}^\beta\times\mathcal{L}$ for each $\beta\in(0,1]$ solves our desired BSDE,

 $$\begin{array}{ll}
   Y_t=\xi+\int_{t}^{T}f(s,Y_s ,Z_s ,V_s )ds -\int_{t}^{T}Z_s  dB_s-\int_t^T \int_U V_s (e) \tilde{\mu}(ds,de), 0\leq t\leq T.\\
\end{array}$$  

For the general case, it suffices to subdivide the interval time $[0,T]$ into a finite number of small intervals, and using standard arguments, we can show the existence of a solution $(Y,Z,V)$ of BSDE (\ref{s2}) on the whole interval $[0,T]$. $\square$

\end{document}